\documentclass[reqno]{amsproc}
\usepackage{amssymb}

\usepackage{amsmath}

\theoremstyle{plain}
\newtheorem{thm}{Theorem}[section]
\newtheorem{cor}[thm]{Corollary}
\newtheorem{lemma}{Lemma}[section]
\newtheorem{proposition}[thm]{Proposition}

\theoremstyle{remark}
\newtheorem{remark}[thm]{Remark}

\theoremstyle{definition}

\newtheorem*{exx}{Example}

\numberwithin{equation}{section}

\def\qed{{\unskip\nobreak\hfil\penalty50\hskip1em\nobreak\hfil {\vrule
height7ptwidth7ptdepth0pt}
\parfillskip=0pt\finalhyphendemerits=0\par}\vskip3mm}

\newcommand{\la}{\langle}
\newcommand{\ra}{\rangle}

 \newcommand{\al}{\alpha}
  \newcommand{\be}{\beta}

  \newcommand{\vep}{\varepsilon}

\def\cA{\hbox{$\mathcal A$}}

\def\bs{\hbox{$\mathbf s$}}
\def\bt{\hbox{$\mathbf t$}}
\def\bC{\hbox{$\mathbb C$}}

\def\bN{\hbox{$\mathbb N$}}
\def\bZ{\hbox{$\mathbb{Z}$}}
\def\bR{\hbox{$\mathbb R$}}

\def\ds{\hbox{${}^{**}$}}

\def\on{\hbox{${}\|_\omega$}}
\def\mon{\hbox{${}\|_{\sup,1/\omega}$}}

\def\L1o{\hbox{$L^1(\omega)$}}
\def\Mo{\hbox{$M(\omega)$}}
\def\ho{\hbox{$\widehat\omega$}}
 \def\Ao{\hbox{$\cA_{\omega}$}}
 \def\C0o{\hbox{$C_0(1/\omega)$}}
 \def\Lio{\hbox{$L^\infty(1/\omega)$}}
\def\dt{\hbox{$\delta_t$}}
\def\dtn{\hbox{$\delta_{t^n}$}}
\def\dtmn{\hbox{$\delta_{t^{-n}}$}}
\def\dtm{\hbox{$\delta_{t^{-1}}$}}
 
 \def\RUCo{\hbox{$RUC(1/\omega)$}}

\begin{document}

\title[Weak amenability
]{Weak amenability of commutative Beurling algebras} 

\author{Yong Zhang}

\address{ Department of Mathematics\\ 
University of Manitoba\\ \mbox{Winnipeg R3T 2N2}\\ \indent Canada}
\email{zhangy@cc.umanitoba.ca}

\thanks{Supported by NSERC 238949-2011. }

\begin{abstract}
For a locally compact Abelian group $G$ and a continuous weight function $\omega$ on $G$ we show that the Beurling algebra $L^1(G, \omega)$ is weakly amenable if and only if there is no nontrivial continuous group homomorphism $\phi$: $G\to \bC$ such that $\sup_{t\in G}\frac{|\phi(t)|}{\omega(t)\omega(t^{-1})} < \infty$. Let $\widehat\omega(t) = \limsup_{s\to \infty}\omega(ts)/\omega(s)$ ($t\in G$). Then $L^1(G, \omega)$ is $2$-weakly amenable if there is a constant $m> 0$ such that $\liminf_{n\to \infty}\frac{\omega(t^n)\widehat\omega(t^{-n})}{n} \leq m$ for all $t\in G$.
\end{abstract}

\def\timestring{\begingroup
\count0=\time \divide\count0 by 60
\count2 = \count0
\count4 = \time \multiply\count0 by 60
\advance\count4 by -\count0
\ifnum\count4 <10 \toks1={0}%
\else \toks1 = {}%
\fi
\ifnum\count2<12 \toks0 = {a.m.}%
\else \toks0 = {p.m.}%
\advance\count2 by -12
\fi
\ifnum\count2=0 \count2 = 12
\fi
\number\count2:\the\toks1 \number\count4
\thinspace \the\toks0
\endgroup}

\def\date{ \timestring \enskip {\ifcase\month\or January\or February\or
March\or April\or May\or June\or July\or August\or September\or October\or
November\or December\fi}
\number\day}

\maketitle
\section{Introduction}
Let $G$ be a locally compact group. The integral of a function $f$ on a measurable subset $K$ of $G$ against a fixed left Haar measure will be denoted by $\int_K{f}dx$.  A \emph{weight} on $G$ is a positive valued continuous function $\omega$ on $G$ that satisfies $\omega(st) \leq \omega(s)\omega(t)$ for all $s,t\in G$. let $L^1(G)$ and $M(G)$ be, respectively, the usual convolution group algebra and measure algebra of $G$. Consider 
\[ L^1(G, \omega) =\{f:\, f\omega \in L^1(G)\},
\]
 where $f\omega$ denotes the pointwise product of $f$ and $\omega$. In our discussion, most of time $G$ is fixed. So we will normally write $\L1o$ for $L^1(G,\omega)$. Equipped with the norm 
 \[\|f\|_\omega = \int_G{|f(t)|\omega(t)}dt \quad (f\in \L1o) \]
and with the convolution product, $\L1o$ is a Banach algebra. When $\omega \equiv 1$ this is just the usual group algebra $L^1(G)$. 

Let $\cA$ be a Banach algebra and let $X$ be a Banach $\cA$-bimodule. A \emph{derivation} $D$: $\cA\to X$  is a linear mapping from $\cA$ into $X$ that satisfies $D(ab) = aD(b) + D(a)b$ ($a,b\in \cA$). For each $x\in X$ the mapping $a\mapsto ax-xa$ ($a\in \cA$) is a continuous derivation, called an \emph{inner derivation}. 
The Banach algebra $\cA$ is called \emph{amenable} if each continuous derivation from $\cA$ into the dual module $X^*$ is inner for every Banach $\cA$-bimodule $X$. 
The Banach algebra $\cA$ is called \emph{weakly amenable} if every continuous derivation from $\cA$ into $\cA^*$ is inner, and $\cA$ is \emph{$n$-weakly amenable} for an integer $n>0$ if every continuous derivation from $\cA$ into $\cA^{(n)}$, the $n$-th dual of $\cA$, is inner. If $\cA$ is $n$-weakly amenable for each $n>0$ then it is called \emph{permanently weakly amenable}. We refer to the monograph \cite{DAL} for the background and history of these notions.

 It is well-known that the group algebra $L^1(G)$ is always weakly amenable \cite{JOH_weak}. 
In fact, $L^1(G)$ is permanently weakly amenable for any locally compact group $G$ \cite{C-G-Z}.

For Beurling algebras, N. Gr\o nb\ae k \cite{GRO} showed that $L^1(G,\omega)$ is amenable if and only if $G$ is amenable and the function $\omega(t)\omega(t^{-1})$ is bounded on $G$. Weak amenability of $\L1o$ was first studied by W.G. Bade, P.C. Curtis ans H.G. Dales in \cite{B-C-D}, where they showed for the additive group $\bZ$ of all integers and for the weight $\omega_\al(x) = (1+|x|)^\al$ on $\bZ$, $L^1(\bZ, \omega_\al)$  is weakly amenable if and only if $0\leq \al < \frac{1}{2}$. Following this work, Gr\o nb\ae k showed in \cite{GRO-weak} that $L^1(\bZ, \omega)$ is weakly amenable if and only if $\liminf_{n\to \infty}\frac{\omega(n)\omega(-n)}{n} = 0$. Recently E. Samei \cite{SAM} (also see \cite{G-Zab}) showed that for a commutative group $G$, if $\liminf \frac{\omega(t^n)\omega(t^{-n})}{n} = 0$ for all $t\in G$ then $\L1o$ is weakly amenable. For $2$-weak amenability H.G. Dales and A. T.-M. Lau showed in \cite{D-L} that $L^1(\bZ, \omega_\al)$ is $2$-weakly amenable if and only if $0\leq \al < 1$ and that the same is also true for $L^1(\bR, \omega_\al)$. They conjectured that for an Abelian group $G$, 
 $\L1o$ is $2$-weakly amenable whenever  $\liminf_{n\to \infty}\frac{\omega(t^n)}{n} = 0$ for all $t\in G$, after showing that this is true if $\omega$ is almost invariant in the sense that $\lim_{t\to \infty}\sup_{s\in K}|\frac{\omega(st)}{\omega(t)}-1|=0$ for each compact set $K\subset G$. The last result was improved in 
\cite{SAM}
, where the almost invariance condition was replaced by the weaker condition that the function $\widehat\omega$ defined by $\widehat\omega(s) = \limsup_{t\to \infty}\frac{\omega(ts)}{\omega(t)}$ is bounded on $G$. Related to 2-weak amenability of $\L1o$, we note that if $G$ is Abelian, $\L1o$ is semisimple \cite{BhDe} and so, by the Singer-Wermer Theorem \cite[2.7.20]{DAL}, zero is the only continuous derivation on $\L1o$.

In this paper we study weak amenability and $2$-weak amenability for commutative Beurling algebras.

 In Section~\ref{sect WA} we show that a commutative Beurling algebra $\L1o$ is weakly amenable if and only if there is no nontrivial continuous group homomorphism $\phi$: $G\to \bC$ (note that such homomorphisms are called characters in \cite[24.33]{HAR}) such that $\sup_{t\in G}\frac{|\phi(t)|}{\omega(t)\omega(t^{-1})} < \infty$. With this characterization we may easily derive some well-known results obtained in \cite {B-C-D, GRO, SAM} on the weak amenability of commutative Beurling algebras. We will also study special cases in the section. For example, we will explore the weak amenability of $L^1(G,\omega)$ when $G$ is the additive group of the real line $\bR$, when $G$ is the product group of two or several factors and when $\L1o$ is the tensor product of two Beurling algebras.

In Section~\ref{sect 2-WA} we show that $\L1o$ is $2$-weakly amenable if there is a constant $m> 0$ such that $\liminf_{n\to \infty}\frac{\omega(t^n)\widehat\omega(t^{-n})}{n} \leq m$ for all $t\in G$. This result covers several known results on the $2$-weak amenability of commutative Beurling algebras. We will also give an example of a $2$-weakly amenable $\L1o$ for which $\widehat\omega$ is unbounded.

In Section~\ref{Quest} we will discuss some open problems on weak amenability for Beurling algebras.

\section{Preliminaries} 
Given a Banach space $X$, its dual space will be denoted by $X^*$. The action of $f\in X^*$ at $x\in X$ will be denoted either by $f(x)$ or by $\la x, f\ra$.

Let $\cA$ be a Banach algebra and let $X$ be a Banach $\cA$-bimodule. The module action of $\cA$ on $X$ will be denoted by ``$\cdot$''. But if no confusion may occur, we will simply write $ax$ or $xa$ instead of $a\cdot x$ or $x\cdot a$  ($a\in \cA$, $x\in X$). As well known,  the dual space $X^*$ of $X$ is a Banach $\cA$-bimodule with the natural module actions defined by 
\[ \la x, a\cdot f\ra = \la xa,f\ra, \quad \la x, f\cdot a\ra = \la ax,f\ra \]
for $a\in \cA$, $f\in X^*$ and $x\in X$.
In particular, $\cA^*$ is a Banach $\cA$-bimodule.
The bidual space $\cA^{**}$ of $\cA$ may be equipped with two Arens products $\Box$ and $\diamond$, respectively defined by
\[ \la f, u\Box v\ra = \la v\cdot f, u\ra, \quad \la a, v\cdot f\ra = \la fa, v\ra \] 
and
\[ \la f, u\diamond v\ra = \la f\cdot u, v\ra, \quad \la a, f\cdot u\ra = \la af, u\ra \] 
for $u,v\in \cA^{**}$, $f\in \cA^*$ and $a\in \cA$. With either $\Box$ or $\diamond$ giving the product, $\cA^{**}$ becomes a Banach algebra containing $\cA$ as a closed subalgebra. For any Banach $\cA$-bimodule $X$, $X^*$ is also a Banach left $(\cA^{**}, \Box)$-module and a Banach right $(\cA^{**}, \diamond)$-module (but in general it is not an $\cA^{**}$-bimodule no matter $\Box$ or $\diamond$ is used for the product of $\cA^{**}$). The corresponding module actions are given by
\[ \la x, u\cdot f\ra = \la f\cdot x, u\ra, \text{ where } f\cdot x\in \cA^*, \; \la a, f\cdot x\ra = \la xa, f\ra \quad (a\in \cA) \]
and 
\[ \la x, f\cdot u\ra = \la x\cdot f, u\ra, \text{ where } x\cdot f\in \cA^*, \; \la a, x\cdot f\ra = \la ax, f\ra \quad (a\in \cA) \]
for $u\in \cA^{**}$, $f\in X^*$ and $x\in X$. For any $u\in \cA^{**}$ we denote by $\ell_u$ and $r_u$ respectively the left multiplier operator and the right multiplier operator on $X^*$ defined by $\ell_u(f) = u\cdot f$ and $r_u(f) = f\cdot u$ ($f\in X^*$). If $\cA$ has a bounded approximate identity $(e_\al)$, we may take a weak* cluster point $E$ of $(e_\al)$ in $\cA^{**}$. Then $\ell_E$ and $r_E$ are $\cA$-bimodule morphisms on $X^*$. 

Let $G$ be a locally compact group and $\omega$ be a weight on it. The dual space of $\L1o$ may be identified with 
\[ \Lio = L^\infty(G, 1/\omega) =\{f: {f}/{\omega}\in L^\infty(G)\} \]
with the norm given by 
\[  \|f\mon = \text{ess}\sup_{t\in G}\left|\frac{f(t)}{\omega(t)}\right| \quad (f\in \Lio). \]

A function $f\in \Lio$ is called \emph{right $\omega$-uniformly continuous} if the mapping $t \mapsto R_t(f)$ is continuous from $G$ into $\Lio$, where $R_t$ denotes the right translation by $t$, i.e. $R_t(f)(s) = f(st)$ ($s\in G$). The space of all right $\omega$-uniformly continuous functions is denoted by $RUC(G,1/\omega)$ (or abbreviated $\RUCo$). It is well-known (see \cite[Proposition~7.17]{D-L} for example) that 
\[\RUCo = \L1o\cdot \Lio. \]

Denote by $C_{00}(G)$ the space of all compactly supported continuous functions on $G$. The closure of $C_{00}(G)$ in $\Lio$ is $\C0o$ which is a Banach $\L1o$-submodule of $\Lio$. The dual space of $\C0o$ is $\Mo$, the space of all complex regular Borel measures $\mu$ on $G$ that satisfy
\[ \|\mu\on = \int_G \omega(t) d|\mu|(t) < \infty, \]
where $|\mu|$ denotes the total variation measure of $\mu$. $\|\mu\on$ is indeed the norm of $\mu$ in $\Mo$.
With the convolution product of measures, which is denoted by $*$, $\Mo$ is a Banach algebra containing $\L1o$ as a closed ideal. In fact, $\Mo$ is the multiplier algebra of $\L1o$ \cite{GAU}.

Let $X$ be a Banach space. Denote the space of all bounded linear operators on $X$ by $B(X)$. The \emph{strong operator topology} (or briefly $so$-topology) on $B(X)\times B(X)$ is the topology induced by the family of seminorms $\{p_x: x\in X\}$, where
\[  p_x(S, T) = \max\{\|S(x)\|, \|T(x)\|\} \quad (S,T\in B(X)) \]
(see \cite[page 327]{DAL}). Indeed, $B(X)$ is a Banach algebra with the operator norm topology and the composition product. So is $B(X)\times B(X)$. 
As the multiplier algebra of $\L1o$, $\Mo$ is actually regarded as a subalgebra of $B(\L1o) \times B(\L1o)$ with each $\mu\in \Mo$ being identified with $(\ell_\mu, r_\mu) \in B(\L1o) \times B(\L1o)$.  See \cite{DAL} for details.

\begin{lemma}\label{so dense}
Let $G$ be a locally compact group and let $\omega$ be a weight on it. Then $lin\{\dt: t\in G\}$, the linear space generated by the point measures $\dt$ ($t\in G$), is dense in $\Mo$ in the $so$-topology. In particular, for each $h\in \L1o$ there is a net $(u_\al)\subset lin\{\dt: t\in G\}$ such that $\|(u_\al - h)*a\on \to 0$ and $\|a*(u_\al - h)\on \to 0$ for all $a\in \L1o$.
\end{lemma}
\begin{proof}
Denote $V = lin\{\dt: t\in G\}$. It is evident that $V\subset \Mo$.
Let $\mu \in \Mo$. We show that there is a net $(\mu_\al)\subset V$ such that 
 $\|\mu_\al *a -\mu *a\on \to 0$  and  $\|a*\mu_\al -a*\mu\on \to 0$
for every $a\in \L1o$.  

By \cite[Proposition~3.3.41(i)]{DAL} $V$ is dense in $M(G)$ in the $so$-topology. Since $u: =\omega \mu \in M(G)$, there is a net $(u_\al)\subset V$ such that \[ \|u_\al *g -u *g\|_1 \to 0 \text{ and } \|g*u_\al -g*u\|_1 \to 0  \]
 for every $g\in L^1(G)$. Let $\mu_\al = \frac{1}{\omega}u_\al$ and $g = \omega a$. Then $\mu_\al$ still belongs to $V$ and $g\in L^1(G)$. We have
\[ \|\mu_\al *a -\mu*a\on \leq \|u_\al *g - u*g\|_1 \to 0 \]
 and 
 \[  \|a*\mu_\al  - a*\mu\on \leq \|g*u_\al  - g*u\|_1 \to 0. \]
 Thus, $\mu \in so$-cl$(V)$.
This is true for every $\mu\in \Mo$. The proof is complete.
\end{proof}

\section{Weak amenability}\label{sect WA}

We denote the additive group of complex numbers (with the usual metric topology) by $\bC$ and denote by $\bR$ the closed subgroup of $\bC$ consisting of all real numbers.

\begin{thm}\label{weak}
Let $G$ be a locally compact Abelian group and $\omega$ be a weight on $G$. Then $\L1o$ is weakly amenable if and only if there exists no nontrivial continuous group homomorphism $\phi$: $G \to \bC$ such that 
\begin{equation}\label{phi} \sup_{t\in G} \frac{|\phi(t)|}{\omega(t)\omega(t^{-1})} < \infty. \end{equation}
\end{thm}
\begin{proof}
If $\L1o$ is not weakly amenable, then there is a nonzero continuous derivation $D$: $\L1o \to \Lio$. It is standard (see \cite{JOH} for example) that one can extend $D$ to a derivation, still denoted by $D$, from $\Mo$ to $\Lio$.  Define $\Delta(t) = \dtm \cdot D(\dt)$ ($t\in G$). Then $\Delta$ satisfies
\begin{align}
\Delta(t_1t_2) = \delta_{t_2^{-1}t_1^{-1}}\cdot D(\delta_{t_1}\delta_{t_2}) &= \delta_{t_1^{-1}}\cdot D(\delta_{t_1}) + \delta_{t_2^{-1}}\cdot D(\delta_{t_2})\notag \\
                                                        & = \Delta(t_1) + \Delta(t_2) \label{Del}
 \end{align}
for $t_1, t_2 \in G$, and $\Delta(e) = 0$, where $e$ is the unit of $G$. 

We note that $D$ is $so$-weak* continuous. In fact, since $\L1o$ has a bounded approximate identity, by Cohen's Factorization Theorem every $f\in \L1o$ may be written as $f= f_1*f_2$ for some $f_1, f_2\in \L1o$. So, if $\mu_\al \overset{so}{\to} \mu$ in $\Mo$, then 
\begin{align*} 
\lim_\al \la f, D(\mu_\al) \ra &= \lim_\al \la f_1, D(f_2*\mu_\al)\ra - \lim_\al \la \mu_\al*f_1, D(f_2)\ra \\
                                            &= \la f_1, D(f_2*\mu)\ra - \la \mu*f_1, D(f_2)\ra = \la f, D(\mu) \ra.
\end{align*}
This clarifies the $so$-weak* continuity of $D$.
Since $span\{\dt: t\in G\}$ is dense in $\Mo$ in the $so$-topology (Lemma~\ref{so dense}), $\Delta$ is a nontrivial mapping from $G$ to $\Lio$. So there is $h\in \L1o$ such that $\phi(t) = \la h, \Delta(t) \ra$ is a nontrivial complex valued function on $G$. By (\ref{Del}) the function $\phi$ is clearly a group homomorphism from $G$ to $\bC$. It is also continuous. To see this (due to again Cohen's Factorization Theorem) we write  $h = h_1*h_2$ for some $h_1,h_2 \in \L1o$. Then
\[ \phi(t) = \la h_1, D(h_2)\ra - \la\dt*h_1, D(h_2*\dtm) \ra , \]
which is clearly continuous in $t$.  Moreover, 
\[ |\phi(t)| \leq (\|D\| \|h\on) \omega(t)\omega(t^{-1}) \quad (t \in G). \]
 Thus $\phi$: $G\to \bC$ is a nontrivial continuous group homomorphism and it satisfies (\ref{phi}).

For the converse, we assume $\phi$: $G \to \bC$ is a continuous nontrivial group homomorphism that satisfies 
\[ \sup_{t\in G} \frac{|\phi(t)|}{\omega(t)\omega(t^{-1})} \leq m \]
for some $m < \infty$.
 Fix a compact neighborhood $B$ of $e$ in $G$. For each $h\in \L1o$ 
 we define 
\begin{equation}\label{D}
  D(h)(t) = \int_B{\phi(t^{-1}\xi)h(t^{-1}\xi)}d\xi \quad (t\in G). 
\end{equation}
It is standard to check that $D(h)(t)$ is continuous (and hence is measurable) on $G$. Since
\[  \left|\frac{D(h)(t)}{\omega(t)}\right|\leq m 
\int_B{
\omega(\xi^{-1})\omega(t^{-1}\xi)|h(t^{-1}\xi)}|d\xi\leq ml\|h\on \]
for all $t\in G$, we derive $D(h)\in \Lio$ for each $h\in \L1o$, where 
 \[
l = \sup\{\omega(s^{-1}): s\in B \}.
\]
 The mapping $h\mapsto D(h)$ is clearly a nonzero bounded linear mapping from $\L1o$ to $\Lio$. We show it is indeed a derivation. Let $a, b\in \L1o$. Then
\begin{align*}
 D(a*b)(t) &= \int_B\phi(t^{-1}\xi)\int_G{a(s)b(s^{-1}t^{-1}\xi)}ds d\xi \\
   & = \int_B{\int_G{a(s)\left(\phi(s^{-1}t^{-1}\xi) + \phi(s)\right)b(s^{-1}t^{-1}\xi)}}ds d\xi \\
   & = \int_G{a(s)D(b)(ts)}ds + \int_G{D(a)(st)b(s)}ds \\
   & = [a\cdot D(b) + D(a)\cdot b](t) \quad (t\in G).
\end{align*}
In the above computation we have used the Fubini theorem to exchange the order of integrals. We can do this because $B$ is compact and the supports of $a$ and $b$ are $\sigma$-finite. Therefore $D(a*b) = a\cdot D(b) + D(a)\cdot b$ for all $a,b\in \L1o$, i.e. $D$: $\L1o\to \Lio$ is a nonzero continuous derivation. Thus $\L1o$ is not weakly amenable.
\end{proof}

\begin{remark}\label{IN}
If $G$ is an IN group, we take any compact neighborhood $B$ of $e$ in $G$ such that $sBs^{-1} = B$ for all $s\in G$. Then the argument for the necessity part of Theorem~\ref{weak} may be adapted to show the following: If $\L1o$ is weakly amenable then there is no continuous group homomorphism $\phi$: $G \to \bC$ such that $\phi$ is not trivial on $B$ and such that (\ref{phi}) holds. To see this we first note $\int_B{f(s\xi)}d\xi = \int_B{f(\xi s)}d\xi$ for $f\in L^1(G,\omega)$ and $s\in G$. This property ensures that the mapping $D$ defined by (\ref{D}) is still a continuous derivation, assuming the above $\phi$ exists. If $D$ is inner then it must be trivial at all $h$ belonging to the center $Z\L1o$ of $\L1o$.  However, $h_\phi := \overline\phi \chi_B\in Z\L1o$ and
\[
D(h_\phi)(t) = \int_{B\cap tB}{|\phi(t^{-1}\xi)|^2}d\xi = \int_{B\cap t^{-1}B}{|\phi(\xi)|^2}d\xi
\]
is nontrivial if $\phi$ is not trivial on $B$,
where $\overline\phi$ is the conjugate of $\phi$ and $\chi_B$ is the characteristic function of $B$. So $D$ is not inner and thus $\L1o$ is not weakly amenable.

\end{remark}

\begin{remark}
For a discrete Abelian group $G$, Theorem~\ref{weak} was obtained by Gr\o nb\ae k in \cite{GRO-weak}. As indicated there, when $G = \bZ$ (the discrete additive group of all integers) all group homomorphisms from $\bZ$ to $\bC$ are of the form $\phi(n) = nc_0$ ($n\in \bZ$, $c_0\in \bC$). Therefore, for any weight $\omega$ on $\bZ$, $\ell^1(\bZ, \omega)$ is weakly amenable if and only if $\sup_{n\in \bN} \frac{n}{\omega(n)\omega(-n)} = \infty$, or equivalently, $\inf_{n\in \bN} \frac{\omega(n)\omega(-n)}{n} = 0$.
\end{remark}

The above argument certainly works also for $G = \bR$. But we have more to say  for $\bR$ later in Corollary ~\ref{cor R}.

\begin{remark}
Since a commutative Banach algebra is permanently weakly amenable if and only if it is weakly amenable \cite{D-G-G}, the condition in Theorem~\ref{weak} is also a necessary and sufficient condition for $\L1o$ to be permanently weakly amenable.
\end{remark}

When one applies Theorem~\ref{weak}, it suffices to consider only real valued group homomorphisms. Precisely we have the following theorem.

\begin{thm}\label{real homo}
Let $G$ be a locally compact Abelian group and $\omega$ be a weight on $G$. Then $\L1o$ is weakly amenable if and only if there exists no nontrivial continuous group homomorphism $\phi$: $G \to \bR$ such that (\ref{phi}) holds.
\end{thm}
\begin{proof}
The necessity is trivial. For the sufficiency, suppose that $\L1o$ is not weakly amenable. Then, by Theorem~\ref{weak}, there is a continuous complex valued nonzero homomorphism $\phi$ such that (\ref{phi}) holds. The real part $\phi_r$ and the imaginary part $\phi_i$ of $\phi$ are both still continuous group homomorphisms, they satisfy the same inequality (\ref{phi}), and they are real valued. If $\phi \neq 0$ then at least one of $\phi_r$ and $\phi_i$ is nonzero. So there exists a nontrivial continuous real valued group homomorphism such that (\ref{phi}) holds. 
\end{proof}

We now consider some special cases to illustrate how Theorem~\ref{weak} applies.
\begin{cor}[\cite{SAM}]
Let $G$ be a locally compact Abelian group and $\omega$ be a weight on $G$. If for each $t\in G$ 
\begin{equation}\label{inf} 
\inf_{n\in \bN} \frac{\omega(t^n)\omega(t^{-n})}{n} = 0,
\end{equation}
then $\L1o$ is weakly amenable.
\end{cor}
\begin{proof}
Let $\phi$: $G\to \bR$ be any nontrivial group homomorphism and let $s\in G$ be such that $\phi(s) \neq 0$. We have $\phi(s^n) = n\phi(s)$ ($n\in \bN$). If (\ref{inf}) holds for $t=s$, then
\[ \sup_{t\in G} \frac{|\phi(t)|}{\omega(t)\omega(t^{-1})} \geq \sup_{n\in \bN} \frac{|\phi(s^n)|}{\omega(s^n)\omega(s^{-n})} =\sup_{n\in \bN} \frac{|\phi(s)|n}{\omega(s^n)\omega(s^{-n})} = \infty. \]
So (\ref{phi}) does not hold for any such nonzero homomorphism $\phi$. By Theorem~\ref{real homo}, $\L1o$ is weakly amenable.
\end{proof}

\begin{cor}\label{cor R}
Let $\omega$ be a weight on $\bR$. Then the following statements are equivalent.
\begin{enumerate}
\item The Beurling algebra $L^1(\bR, \omega)$ is weakly amenable. \label{R weak}
\item $\limsup_{t\to \infty} \frac{|\phi(t)|}{\omega(t)\omega(-t)} = \infty$ for each nonzero continuous group homomorphism $\phi$: $\bR \to \bC$.\label{R phi}
\item $\liminf_{t\to \infty} \frac{\omega(t)\omega(-t)}{|t|} = 0$. \label{limit t}
\item $\liminf_{n\to \infty} \frac{\omega(nt)\omega(-nt)}{n} = 0$ for all $t \in \bR$. \label{limit nt}
\item $\liminf_{n\to \infty} \frac{\omega(n)\omega(-n)}{n} = 0$.\label{limit n}
\item There is $t_0\in \bR$ such that $t_0 \neq 0$ and \mbox{$\liminf_{n\to \infty} \frac{\omega(nt_0)\omega(-nt_0)}{n} = 0$.\label{limit t0}}
\end{enumerate}
\end{cor}

\begin{proof}
The equivalence of (\ref{R weak}) and (\ref{R phi}) follows straightforward from Theorem~\ref{weak}.  

(\ref{R phi})$\Rightarrow$(\ref{limit t}):  Simply consider $\phi(t) = t$ ($t\in \bR$). We see immediately that (\ref{limit t}) holds if (\ref{R phi}) is true.

(\ref{limit t})$\Rightarrow$(\ref{R phi}): 
Given a nonzero continuous group homomorphism $\phi$: $\bR \to \bC$, it is a well-known fact that there is $z_0\in \bC$,  $z_0 \neq 0$, such that  $\phi(t) = tz_0$. Thus,
\[ \frac{|\phi(t)|}{\omega(t)\omega(-t)} = |z_0|\frac{|t|}{\omega(t)\omega(-t)} . \]
This relation shows that (\ref{R phi}) is the case if (\ref{limit t}) holds.

(\ref{limit t})$\Rightarrow$(\ref{limit nt}): If $t=0$, the limit in (\ref{limit nt}) is trivially true. If $t\neq 0$, without loss of generality, we may assume $t=t_0 > 0$. If (\ref{limit t}) holds then there is a positive sequence $(t_i)\subset \bR$ such that $t_i \to \infty$ and
\[ \lim_{i\to \infty} \frac{\omega(t_i)\omega(-t_i)}{t_i} = 0. \]
Take $n_i\in \bN$ and $0\leq s_i <t_0$ such that $t_i = n_it_0 + s_i$. We have
\[ \frac{\omega(t_i)\omega(-t_i)}{t_i} \geq \frac{1}{(t_0 + s_i/n_i)\omega(s_i)\omega(-s_i)} \frac{\omega(n_i t_0)\omega(-n_i t_0)}{n_i}. \]
Since $0\leq s_i < t_0$, $n_i \to \infty$ and $\frac{1}{(t_0 + s_i/n_i)\omega(s_i)\omega(-s_i)}$ is uniformly bounded away from $0$ as $i\to \infty$. The above inequality leads to 
\[ \lim_{i\to \infty} \frac{\omega(n_it_0)\omega(-n_it_0)}{n_i} = 0.\]
This shows that (\ref{limit nt}) holds when $t = t_0$ for all $t_0 > 0$. Therefore it holds for all $t\in \bR$.

(\ref{limit nt})$\Rightarrow$ (\ref{limit n}), (\ref{limit n})$\Rightarrow$(\ref{limit t0}) and (\ref{limit t0})$\Rightarrow$(\ref{limit t}) are trivial. The proof is complete.


\end{proof}

Let $H$ and $R$ be two locally compact Abelian groups. We consider the product group $H\times R=\{(s,t): s\in H, t\in R\}$. With the product topology it is a commutative locally compact group. We may regard $H$ and $R$ as closed subgroups of $H\times R$, identifying $s$ with $(s, e_R)$ and $t$ with $(e_H, t)$ for $s\in H$ and $t\in R$, where $e_H$ and $e_R$ are identities of $H$ and $R$ respectively. Let $\omega$ be a weight on $H\times R$. Then $\omega_H = \omega|_{H}$ and $\omega_R = \omega|_{R}$ are weights on $H$ and $R$ respectively. Following \cite{GRO} we denote the symmetrization of $\omega$ by $\Omega$, that is,
 $\Omega(s,t) = \omega(s,t)\omega(s^{-1}, t^{-1})$ ($s\in H$, $t\in R$).
\begin{thm}\label{product}
If both $L^1(H, \omega_H)$ and $L^1(R, \omega_R)$ are weakly amenable then so is $L^1(H\times R, \omega)$. Conversely, the algebra \mbox{$L^1(H\times R, \omega)$} is not weakly amenable in any of the following conditions:
\begin{enumerate}
\item $L^1(H, \omega_H)$ is not weakly amenable and $\sup_{(s,t)\in H\times R} \frac{\Omega(s,e_R)}{\Omega(s,t)} < \infty$.\label{G1}
\item $L^1(R, \omega_R)$ is not weakly amenable and $\sup_{(s,t)\in H\times R} \frac{\Omega(e_H, t)}{\Omega(s,t)} < \infty$.\label{G2}
\end{enumerate}
\end{thm}
\begin{proof}
If $\phi$: $H\times R \to \bC$ is a nonzero continuous group homomorphism, then either $\phi|_{H}$ or $\phi|_{R}$ is nonzero. If (\ref{phi}) holds for $G = H\times R$, then it holds for $G=H$ and for $G=R$. Thus at least one of $L^1(H, \omega_H)$ and $L^1(R, \omega_R)$ is not weakly amenable if $L^1(H\times R, \omega)$ is not weakly amenable. This shows the first assertion of the theorem.
 
For the second assertion, suppose that (\ref{G1}) is the case (the proof for the other case is similar). Then there is nonzero continuous group homomorphism $\phi$: $G=H \to \bC$ such that (\ref{phi}) holds. Let $\phi'(s,t) = \phi(s)$ ($s\in H$, $t\in R$). $\phi'$ is a nonzero continuous group homomorphism from $H\times R$ to $\bC$ and 
\[ \frac{|\phi'(s,t)|}{\Omega(s,t)} = \frac{|\phi(s)|}{\Omega(s,e_R)}\frac{\Omega(s,e_R)}{\Omega(s,t)} \leq l \frac{|\phi(s)|}{\omega_H(s)\omega_H(s^{-1})}, \]
where $l$ is a constant such that $\sup_{(s,t)\in H\times R} \frac{\Omega(s,e_R)}{\Omega(s,t)} \leq l$.
So (\ref{phi}) holds for $\phi'$ and $G = H\times R$. Therefore, $L^1(H\times R, \omega)$ is not weakly amenable from Theorem~\ref{weak}.

\end{proof}

\begin{exx}
Consider the polynomial weight $\omega(s,t) = (1+|s|+|t|)^\al$ on $\bR^2 = \bR \times \bR$ ($\al>0$). Then $\omega_H = \omega_R = \omega_\al$. From Corollary~\ref{cor R} we see $L^1(\bR,\omega_\al)$ is weakly amenable if and only if $\al < 1/2$. Since $\omega(s,0)\leq \omega(s,t)$ for $s,t\in \bR$, we have $\Omega(s,0)/\Omega(s,t) \leq 1$. Therefore, the inequality in (\ref{G1}) (and (\ref{G2})) of Theorem~\ref{product} holds. From Theorem~\ref{product} we immediately derive that $L^1(\bR^2,\omega)$ is weakly amenable if and only if $\al <1/2$.
\end{exx}

We now discuss some consequences of Theorem~\ref{product}.

\begin{cor}
Let $\omega$ be a weight on the the additive group $\bR^n$. Denote $e_i = (0,\cdots,1,0,\cdots, 0)$ ($i=1,2,\cdots,n$), where $1$ appears only at $i$-th coordinate, and let $\omega_i$ be the weight on $\bR$ defined by $\omega_i(t) = \omega(te_i)$ ($t\in \bR$). If
\[
\liminf_{n\to \infty} \frac{\omega_i(n)\omega_i(-n)}{n} = 0
\]
for all $i=1,2,\cdots, n$, then  $L^1(\bR^n,\omega)$ is weakly amenable.
\end{cor}
\proof  From  Corollary~\ref{cor R}(\ref{limit n}) $L^1(\bR,\omega_i)$ is weakly amenable for each $i$. Then Theorem~\ref{product}  applies. \qed

\begin{cor}
Let $G_1$ and $G_2$ be two locally compact Abelian groups and let $\omega_1$ and $\omega_2$ be weights on them, respectively. Then the (projective) tensor product algebra $L^1(G_1,\omega_1)\hat\otimes L^1(G_2, \omega_2)$ is weakly amenable if and only if both $L^1(G_1,\omega_1)$ and  $L^1(G_2, \omega_2)$ are weakly amenable.
\end{cor}
\begin{proof}
$L^1(G_1,\omega_1)\hat\otimes L^1(G_2, \omega_2) \cong L^1(G_1\times G_2,\omega_1\times \omega_2)$. For $H= G_1$, $R=G_2$ and $\omega= \omega_1\times \omega_2$ the inequalities in (\ref{G1}) and (\ref{G2}) of Theorem~\ref{product} hold evidently. Thus, the result follows from Theorem~\ref{product}.
\end{proof}

\begin{cor}\label{product of n}
Let $G_1$, $G_2$, \ldots, $G_k$ be $k$ locally compact Abelian groups. Suppose that $\omega$ is a weight on $G_1\times G_2\times \cdots \times G_k$ and there is a constant $r >0$ such that for each $i = 1,2,\cdots, k$
\[  \omega(T_i) \leq r \omega(t_1,t_2,\cdots,t_k) \quad (t_j\in G_j,\, j=1,2,\cdots,k),  \]
where $T_i$ represents for the element of $G_1\times G_2\times \cdots \times G_k$ whose $i$-th coordinate is $t_i$ and each of the other coordinates is the unit element of the corresponding component group. Then $L^1(G_1\times G_2\times \cdots \times G_k, \omega)$ is weakly amenable if and only if all $L^1(G_i, \omega_{G_i})$ ($i= 1,2,\cdots,k$) are weakly amenable. 
\end{cor}
\begin{proof}
Simply apply induction and use Theorem~\ref{product}.
\end{proof}
If $G$ is a compactly generated locally compact Abelian group then it is topologically isomorphic with $\bR^p\times \bZ^q\times F$ for some integers $p\geq 0 , q\geq 0$ and some compact group $F$. We may write such a group as $G=G_1\times G_2\times \cdots \times G_k\times F$, where $G_i$ is either $\bR$ or $\bZ$ ($i= 1,\cdots,k$). Denote $G_i^+ = \{t\in G_i: t\geq 0\}$. We note that, for any compact group $F$, $L^1(F,\omega)$ is isomorphic with $L^1(F)$ and hence is weakly amenable for any weight $\omega$.

\begin{cor}\label{abs.}
Let  $G=G_1\times G_2\times \cdots \times G_k\times F$ be a locally compact, compact generated Abelian group, where $G_i$ is either $\bR$ or $\bZ$ ($i= 1,\cdots,k$) and $F$ is a compact group. Let $\omega$ be a weight on $G$ which can be written in the form
\[ \omega(t_1, t_2, \cdots, t_k, s) = w(|t_1|, |t_2|, \cdots, |t_k|, s), \quad ((t_1, t_2, \cdots, t_k, s)\in G), \]
 where $w(x_1,x_2,\cdots,x_k,s)$ is a function on $G_1^+\times G_2^+\times \cdots \times G_k^+\times F$ which is increasing in each $x_i$ ($i=1,2,\cdots,k$). Then $L^1(G, \omega)$ is weakly amenable if and only if all $L^1(G_i, \omega_{G_i})$ ($i = 1, 2, \cdots, k$) are weakly amenable.
\end{cor}
\begin{proof}
It is readily checked that the condition of Corollary~\ref{product of n} is fulfilled with $r = \sup_{s\in F}\omega(0,0,\cdots, 0, s)$.
\end{proof}

\begin{remark}
In particular, if $\omega$ is a polynomial weight, then Corollary~\ref{abs.} gives \cite[Theorem~7.1(i)]{SAM}.
\end{remark}

\section{$2$-weak amenability}\label{sect 2-WA}

Let $G$ be a locally compact group and let $\omega$ be a weight on $G$. Define
\[  \ho(t) = \limsup_{s\to\infty}\frac{\omega(ts)}{\omega(s)}: = \inf_{K}\sup_{s\in G\setminus K}\frac{\omega(ts)}{\omega(s)}, \]
where the infimum is taken over all compact subsets of $G$. The function $\ho$ is not guaranteed to be continuous although it is indeed submultiplicative. It is even not clear whether $\ho$ is a measurable function on $G$. But we will not use the measurability of $\ho$ in our argument.

\begin{thm}\label{2-w}
Let  $G$ be an locally compact Abelian group and let $\omega$ be a weight on $G$. If there is a constant $m>0$ such that
\begin{equation}\label{hat omega}
  \liminf_{n\to\infty}\frac{\omega(t^n)\ho(t^{-n})}{n} \leq m \quad (t\in G),  
\end{equation}
then $\L1o$ is $2$-weakly amenable. 
\end{thm}
\begin{proof}
As in \cite{D-L}, we use $\Ao$ to denote $\L1o$. Suppose that $D$: $\Ao \to \Ao\ds$ is a continuous derivation. We aim to show that $D = 0$. 

It was proved in \cite{D-L} that, as a consequence of the Singer-Wermer theorem, $D(\L1o)\subset \C0o^\perp$. Here $\C0o$ is regarded as a closed submodule of the $\Ao$-bimodule $\Lio$. Denote the quotient module $\Lio/\C0o$ by $X$. Then $\C0o^\bot \cong X^*$ as a Banach $\Ao$-bimodules. Let $(e_\al)$ be a bounded approximate identity of $\Ao$ and let $E$ be a weak* cluster point of $(e_\al)$ in $\Ao\ds$. We then have
\[ X^* = r_E(X^*)\oplus (I-r_E)(X^*), \]
where $I$ denotes the identity operator on $X^*$. 
It is evident that $r_E(X^*)\cong (\Ao X)^*$ and $(I - r_E)(X^*)\cong (X/\Ao X)^*$ as Banach $\Ao$-bimodules (see the proof  of \cite[Proposition~1.8]{JOH}). With this module decomposition the derivation $D$ may be written as the sum of two continuous derivations $D_1=r_E\circ D$: $\Ao\to r_E(X^*)$ and $D_2=(I -r_E)\circ D$: $\Ao\to (I -r_E)(X^*)$. Since the $\Ao$-module actions on $X/(\Ao X)$ is trivial, $D_2 = 0$ according to \cite[Proposition~1.5]{JOH}. Therefore $D = D_1$, that is, $D$: $\Ao\to r_E(X^*)$. Note that $X$ is a symmetric $\Ao$-bimodule. So $\Ao X = \Ao X\Ao$ is a neo-unital $\Ao$-bimodule. As is well-known, we may extend $\Ao$-module actions on $\Ao X$ to $\Mo$-bimodule actions so that $\Ao X$ becomes a unital (symmetric) $\Mo$-bimodule. Moreover, $D$ may be uniquely extended to a continuous derivation from $\Mo$ to $r_E(X^*)$. Thus, for each $t\in G$, $D(\delta_t)$ is well-defined in this sense. We show $D(\dt) = 0$ for all $t\in G$.

Note that $\Ao X = RUC(1/\omega)/\C0o$. Given any $x\in \Ao X$, any compact set $K\subset G$ and any $\vep>0$, there is $f\in RUC(1/\omega)$ such that 
\[ x = [f], \;f|_K =0 \text{ and } \|f\|_{\sup,1/\omega} \leq \|x\| + \vep, \]
 where $[f] = f + \C0o$ represents the coset of $f$ modulo $\C0o$. In fact, from the definition of the quotient norm one may choose $h\in \RUCo$ such that $x = [h]$ and $\|h\mon \leq \|x\| + \vep$. On the other hand, there is $f_0\in C_{00}(G)$ such that $0\leq f_0\leq 1$, $f_0(t) = 1$ for $t\in K$. We take $f = (1-f_0)h$. Then $[f] = [h] = x$ since $f_0h\in C_{00}(G)$. It is easily seen that $\|f\mon \leq \|h\mon \leq \|x\| + \vep$. Then for $t\in G$ we have
\begin{align*}
 \|x\cdot \dt\| &\leq \|f\cdot \dt\mon = \sup_{s\in G} \left|\frac{f(ts)}{\omega(s)}\right|\\
                 & \leq \|f\mon \sup_{s\in G\setminus (t^{-1}K)}\frac{\omega(ts)}{\omega(s)} \leq (\|x\| + \vep) \sup_{s\in G\setminus (t^{-1}K)}\frac{\omega(ts)}{\omega(s)}.
                 \end{align*}
Since $\vep >0$ and $K\subset G$ were arbitrarily given, we derive 
\[ \|x\cdot \dt \| \leq \|x\|\ho(t) \quad (x\in \Ao X,\; t\in G). \]
Then for $\Phi \in (\Ao X)^*$,
\[ |\la x, \dt\cdot \Phi\ra| = |\la x\cdot\dt, \Phi\ra| \leq \|x\| \ho(t)\|\Phi\| \quad (x\in \Ao X).  \]
This implies that $\|\dt\cdot \Phi\| \leq \ho(t) \|\Phi\|$ ($\Phi\in (\Ao X)^*$, $t\in G$).

As $G$ is commutative, for each integer $n$ we have
\[ \dtmn\cdot D(\dtn) = n \dtm \cdot D(\dt) \quad (t\in G). \]
The above discussion allows us to estimate the norm as follows.
\[\|\dtmn\cdot D(\dtn)\| \leq \ho(t^{-n}) \|D(\dtn)\| \leq \omega(t^n)\ho(t^{-n})\|D\|. \]
We then have
\[ \|\dtm \cdot D(\dt)\| = \frac{1}{n} \|\dtmn\cdot D(\dtn)\|\leq \frac{\omega(t^n)\ho(t^{-n})}{n}\|D\|. \]
From the hypothesis we immediately obtain
\[ \|\dtm \cdot D(\dt)\| \leq m\|D\| \quad (t\in G).  \]

Let $\Delta(t) = \dtm \cdot D(\dt)$ ($t\in G$), and let $B = \Delta(G)$. Then $B$ is a bounded subset of $r_E(X^*)$. However, similar to the counterpart that we showed in the proof of Theorem~\ref{weak}, $\Delta(t_1t_2) =  \Delta(t_1) + \Delta(t_2)$
for all $t_1, t_2 \in G$. As a consequence, one sees easily via induction that $\Delta(t^k) = k \Delta (t)$ ($t\in G$). Therefore
\[ \|\Delta(t)\| = \frac{1}{k} \|\Delta(t^k)\| \leq \frac{m\|D\|}{k}  \]
for all integers $k>0$. This shows that $\Delta(t) =0$. Hence $D(\dt) = 0$ for all $t\in G$. This implies that $D(u) = 0$ for $u \in span \{\dt: t\in G\}$. As a continuous derivation from $\Ao$ to $(\Ao X)^*$, $D$ is $so$-weak* continuous. Since $span\{\dt: t\in G\}$ is dense in $\Mo$ in the $so$-topology (Lemma~\ref{so dense}), We finally get $D(u) = 0$ for all $u\in \Mo$. So $D=0$. This shows that $\L1o$ is $2$-weakly amenable.

\end{proof}

\begin{exx}
Consider the additive group $\bZ^2$ and the weight $\omega$ on it defined by
\[  \omega(s, t) = (1+|s|+|t|)^\al(1 + |s+t|)^{\be}, \]
where $\al \geq 0$ and $\be \geq 0$. Then $\widehat \omega(s,t) = (1+|s+t|)^\be$ which is  unbounded if $\be>0$. However, it is readily seen that $\lim \frac{\omega(ns,nt)\widehat\omega(-ns,-nt)}{n} =0$ when $\al + 2\be < 1$. So $\ell^1(\bZ^2,\omega)$ is $2$-weakly amenable if $\al + 2\be < 1$ due to Theorem~\ref{2-w}.
\end{exx}

In general, when a weight is the product of a polynomial weight of order less than 1 and some other weight which does not increase ``too fast'', the corresponding Beurling algebra will be $2$-weakly amenable. Precisely we have the following.

\begin{cor}
Let $G = \bR^m \times \bZ^k$, where $m,k\geq 0$ are integers. Let 
\[  \omega(\bs, \bt) = (1+|\bs|+|\bt|)^\al\omega_0(\bs, \bt) \quad (\bs\in \bR^m, \bt\in \bZ^k) \]
where $0\leq \al <1$, $|\bs|$ and $|\bt|$ denote the Euclidean norm of $\bs$ and $\bt$, and $\omega_0$ is a weight on $G$ satisfying 
\[  \liminf_{n\to \infty}\frac{\omega_0(n\bs, n\bt)}{n^{(1-\al)/2}} = 0 \quad (\bs\in \bR^m, \bt\in \bZ^k).  \]
Then $L^1(G,\omega)$ is $2$-weakly amenable.
\end{cor}
\begin{proof}
It is readily seen that $\widehat\omega = \widehat\omega_0 \leq \omega_0$. So
\[  \omega(\bs, \bt)\widehat\omega(-\bs, -\bt) \leq (1+|\bs|+|\bt|)^\al\omega_0(\bs, \bt)\omega_0(-\bs, -\bt). \]
Therefore
\[ \liminf_{n\to \infty} \frac{\omega(n\bs, n\bt)\widehat\omega(-n\bs, -n\bt) }{n} = 0 \]
for all $\bs\in \bR^m, \bt\in \bZ^k$. By Theorem~\ref{2-w}, $L^1(G,\omega)$ is $2$-weakly amenable.
\end{proof} 
 
\begin{remark}
The product weights discussed in \cite[page 168]{D-L} clearly satisfy the condition of the above corollary. So the $2$-weak amenability of the corresponding Beurling algebras also follows from this corollary.
\end{remark} 

\section{Further comments on open problems}\label{Quest}

\begin{enumerate}
\item[1.] Weak amenability of non-Abelian Beurling algebras

When $G$ is not Abelian and $\omega$ is not trivial, except for the necessary condition discussed in Remark~\ref{IN} for IN groups weak amenability of $L^1(G,\omega)$ is completely unknown. In fact, to the author's knowledge as far, in the literature  there is not even an example of weakly amenable non-Abelian Beurling algebras with a non-trivial weight. 

\item[2.] 2-Weak amenability of Beurling algebras

2-Weak amenability of $L^1(G)$ is closely related to the well-known derivation problem for $L^1(G)$ which asks whether every continuous derivation from $L^1(G)$ into $M(G)$ is inner. The problem has been solved affirmatively in general by V. Losert recently. Derivation problem for a Beurling algebra $L^1(G,\omega)$ is still open and seems not approachable by the method of Losert's. In general, we would like to know when $L^1(G,\omega)$ is 2-weakly amenable.  For Abelian groups $G$,
after our Theorem~\ref{2-w} we would like to know whether condition (\ref{hat omega}) is also necessary for $L^1(G,\omega)$ to be 2-weakly amenable.

\item[3.] Weak amenability of the center algebra of a non-abelian Beurling algebra

The center $ZL^1(G,\omega)$ of $L^1(G,\omega)$ is an Abelian Banach subalgebra of $L^1(G,\omega)$. It is well-known that $ZL^1(G,\omega)$ is not trivial if and only if $G$ is an IN group. Since $ZL^1(G,\omega)=L^1(G,\omega)$ when $G$ is abelian, Studying weak amenability of $ZL^1(G,\omega)$ may be regarded as a natural extension to the investigation of weak amenability for Abelian Beurling algebras. Even for $\omega \equiv 1$, we do not know a full answer to whether  $ZL^1(G,\omega)$ (abbreviated $ZL^1(G)$) is weakly amenable. However, it was shown in \cite[Theorem~2.4]{A-S-S} that $ZL^1(G)$  is hyper-Tauberian if each conjugacy class of $G$ is relatively compact. Consequently, $ZL^1(G)$ is weakly amenable for this type of groups $G$ \cite[Theorem~0.2(i)]{A-S-S}. In particular, $ZL^1(G)$ is weakly amenable if $G$ is compact. We conclude this paper with a simple proof to this last fact. The proof is based on the famous Peter-Weyl theorem.
\end{enumerate}

\begin{proposition}
For every compact group $G$, $ZL^1(G)$ is weakly amenable.
\end{proposition}
\begin{proof}
It is a simple fact that if an Abelian Banach algebra $\cA$ contains a subset $E$ of mutually annihilating idempotents (that is $e^2 = e$ for all $e\in E$, and $e_1e_2=0$ if $e_1, e_2 \in E$ and $e_1\neq e_2$) and if span$(E)$ is dense in $\cA$, then $\cA$ is weakly amenable. Indeed, when $G$ is compact $ZL^1(G)$ has such a subset $E =\{d_\pi \chi_\pi: \pi \in \widehat G\}$ (see \cite[Section 27]{HAR2}), where $\widehat G$ is the dual object of $G$, $d_\pi$ is the dimension of the associated irreducible unitary representation $\pi$, and $\chi_\pi$ is the character of the representation $\pi$. Therefore, $ZL^1(G)$ is weakly amenable if $G$ is compact.
\end{proof}

The author is grateful to the referee for drawing his attention to the hyper-Tauberian property of $ZL^1(G)$ discussed in \cite{A-S-S}. He is also grateful to N. Gr\o nb\ae k for indicating that some results in the paper match well with \cite[Theorem 3.4]{GRO 89}, which concerns with weighted semigroup algebras on the additive semigroup $\bR_+^n$.


\begin{thebibliography}{99}

\bibitem{A-S-S}
A. Azimifard, E. Samei and N. Spronk, Amenability properties of the centres of group algebras, J. Funct. Anal. 256 (2009), 1544-1564. \vskip 2mm

 
\bibitem{B-C-D} W.G. Bade, P.C. Curtis Jr. and H.G. Dales,
Amenability and weak amenability for Beurling and Lipschitz algebras, 
\emph{Proc. London Math. Soc.} {55} (1987), 359--377. \vskip2mm

\bibitem{BhDe}
S.J. Bhatt and H.V. Dedania, A Beurling algebra is semisimple: an elementary proof, Bull. Australian Math. Soc. 66 (2002), 91--93.\vskip2mm


\bibitem{C-G-Z} Y. Choi, F. Ghahramani, and Y. Zhang, Approximate and pseudo-amenability of various classes of Banach algebras, J. Funct. Anal. {256} (2009), 3158-3191. \vskip2mm


\bibitem{DAL} H.G. Dales, \emph{Banach algebras and automatic continuity}, Clarendon Press,
Oxford, (2000).\vskip2mm

\bibitem{D-G-G}
  H. G. Dales, F. Ghahramani and N. Gr\o nb\ae k, 
  \emph{Derivations into iterated duals of Banach algebras},
  Studia Math. {128} (1998), 19--54.\vskip 2mm

\bibitem{D-L}
H.G. Dales and A. T.-M. Lau, The second duals of Beurling algebras, Mem. Amer. Math. Soc. 177 n. 836 (2005).\vskip2mm

\bibitem{GAU}
G. I. Gaudry,
Multipliers of weighted Lebesgue and measure spaces, Proc. Lond, Math. Soc. 19 (1969), 327--340.\vskip2mm

\bibitem{G-Zab} { F. Ghahramani and G. Zabandan}, $2$-weak amenability of a Beurling algebra and amenability of its second dual, Int. J. Pure Appl. Math. {16} (2004) 75--86.\vskip2mm

\bibitem{GRO 89}
N. Gr\o nb\ae k,
{Commutative Banach algebras, module derivations, and semigroups}, J. London Math. Soc. {(2) 40} (1989), 137–157. 
\vskip2mm

\bibitem{GRO-weak}
N. Gr\o nb\ae k,
{A characterization of weakly amenable Banach algebras}, Studia Math.
{94} (1989), 149--162.\vskip2mm

\bibitem{GRO} {N. Gr\o nb\ae k}, { Amenability of weighted convolution algebras 
on locally compact groups}, \emph{Trans. Amer. Math. Soc.} {319} 
(1990), 765--775.\vskip 2mm

\bibitem{HAR} {E. Hewitt and K. A. Ross}, \emph{ Abstract Harmonic Analysis I}, 
2nd Ed., Springer-Verlag, New York (1979).\vskip 2mm

\bibitem{HAR2} {E. Hewitt and K. A. Ross}, \emph{ Abstract Harmonic Analysis II},  Springer-Verlag, Berlin (1970).\vskip 2mm
 
\bibitem{JOH}{B.E. Johnson},{ Cohomology in Banach algebras},
\emph{ Mem. Amer. Math. Soc.} 127 (1972).\vskip2mm

\bibitem{JOH_weak} {B.E. Johnson}, Weak amenability of group algebras, Bull. London Math. Soc. 23 (1991),
281--284.\vskip2mm


\bibitem{SAM}
E. Samei, {Weak amenability and $2$-weak amenability of Beurling algebras}, J. Math. Anal. Appl. 346 (2008), 451--467.\vskip2mm


\end{thebibliography}
\end{document}